\documentclass[a4paper,11pt]{amsart}

\usepackage[left=2cm,top=3cm,bottom=2.5cm,right=2.5cm]{geometry}

\usepackage{amsthm}
\usepackage{amssymb}
\usepackage{amsmath}
\usepackage{mathtools}
\usepackage{pdfpages}
\usepackage{changes}
\usepackage{exscale,cite,color,amsopn}
\usepackage[colorlinks=true,pdftex,unicode=true,linktocpage,bookmarksopen,hypertexnames=true]{hyperref}
\usepackage{aliascnt}
\usepackage{csquotes}
\usepackage{colortbl}
\usepackage{enumitem}
\usepackage{verbatim}
\usepackage{tikz-cd}
\usepackage[capitalise]{cleveref}

\usetikzlibrary{trees}

\renewcommand{\leq}{\leqslant}
\renewcommand{\geq}{\geqslant}

\DeclareMathOperator{\Fix}{Fix}
\DeclareMathOperator{\id}{id}

\DeclareMathOperator{\Ret}{Ret}
\DeclareMathOperator{\Soc}{Soc}

\newcommand{\Z}{\mathbb{Z}}

\newcommand{\Aut}{\operatorname{Aut}}

\newcommand{\End}{\mathrm{End}}

\newcommand{\Sym}{\mathfrak{S}}

\newcommand{\G}{\mathcal{G}}

\newcommand{\subtau}{\underset{\tau}{\ast}}
\newcommand{\subtaucab}[1]{\underset{\tau^{#1}}{\ast}}
\newcommand{\subcab}[1]{\overset{#1}{\ast}}

\makeatletter
\numberwithin{equation}{section}
\numberwithin{figure}{section}
\numberwithin{table}{section}
\newtheorem{thm}{Theorem}[section]
\newtheorem*{thm*}{Theorem}
\newtheorem{lem}[thm]{Lemma}
\newtheorem{cor}[thm]{Corollary}
\newtheorem{pro}[thm]{Proposition}

\theoremstyle{definition}
\newtheorem{defn}[thm]{Definition}
\newtheorem*{defn*}{Definition}

\newtheorem{conjecture}[thm]{Conjecture}

\newtheorem{conventions}[thm]{Conventions}
\newtheorem{convention}[thm]{Convention}
\newtheorem*{convention*}{Convention}
\newtheorem{rem}[thm]{Remark}

\AddToHook{env/pro/begin}{\crefalias{thm}{pro}}
\AddToHook{env/problem/begin}{\crefalias{thm}{problem}}
\AddToHook{env/conjecture/begin}{\crefalias{thm}{conjecture}}
\AddToHook{env/lem/begin}{\crefalias{thm}{lem}}
\AddToHook{env/cor/begin}{\crefalias{thm}{cor}}
\AddToHook{env/defn/begin}{\crefalias{thm}{defn}}
\AddToHook{env/convention/begin}{\crefalias{thm}{convention}}
\AddToHook{env/rem/begin}{\crefalias{thm}{rem}}

\makeatother

\title{Proof of Rump's Retraction Conjecture for Quasilinear Cycle Sets}
\author{Carsten Dietzel}
\date{\today}

\address[Carsten Dietzel]{Department of Mathematics and Data Science, Vrije Universiteit Brussel, Pleinlaan 2, 1050 Brussel, Belgium}
\email{Carsten.Dietzel@vub.be}

\begin{document}

\begin{abstract}
    Nondegenerate cycle sets were introduced by Rump as an algebraic framework for nondegenerate, involutive solutions to the Yang--Baxter equation. Nondegenerate cycle set structures on abelian groups, such as translation-invariant and quasilinear cycle sets, are of particular interest when studying the retraction problem in the theory of the Yang--Baxter equation. In this article, we solve the retraction problem for finite quasilinear cycle sets by showing that each nontrivial quasilinear cycle set is retractable, thus proving a conjecture of Rump.
\end{abstract}

\maketitle

\section{Introduction} \label{sec:introduction}

Let $V$ be a finite-dimensional vector space over a field $k$ and let $R: V^{\otimes 2} \to V^{\otimes 2}$ be an automorphism of $k$-vector spaces. Then $V$ is said to satisfy the (linear) \emph{quantum Yang--Baxter equation} if the following equality holds in $\End_k(V^{\otimes 3})$.
\begin{equation} \label{eq:linear_QYBE}
    (R \otimes \id_V) (\id_V \otimes R) (R \otimes \id_V) = (\id_V \otimes R) (R \otimes \id_V) (\id_V \otimes R). \tag{YBE1}
\end{equation}
This braiding condition originated from the work of the physicists James B. McGuire \cite{mcguire}, Chen-Ning Yang \cite{Yang_YB} and Rodney J. Baxter \cite{Baxter_YB} on integrable systems but was soon discovered to have applications in geometric topology by Turaev \cite{Turaev_links} who developed a framework to systematically construct knot invariants from solutions to the Yang--Baxter equation, thus providing a vast generalization to the famous Jones polynomial \cite{jones_polynomial}.

In 1990, Drinfeld proposed studying \emph{set-theoretic solutions} to the Yang--Baxter equation \cite{Drinfeld_Problems}. These are given by a finite set $X$ and a bijection $r: X^2 \to X^2$ that satisfies the following equality of bijections on $X^3$.
\begin{equation} \label{eq:set_QYBE}
    (r \times \id_X) (\id_X \times r) (r \times \id_X) = (\id_X \times r) (r \times \id_X) (\id_X \times r).\tag{YBE2}
\end{equation}
Linearizing such a solution yields a solution to \eqref{eq:linear_QYBE} which can be submitted to a deformation process.

In the middle of the 1990s, Gateva-Ivanova and van den Bergh \cite{GIVdB_IType}, and Etingof, Schedler and Soloviev \cite{ESS_YangBaxter} successfully developed a systematic approach to the investigation of nondegenerate involutive set-theoretic solutions to the Yang--Baxter equation by means of group-theoretic methods, thus initiating the algebraic theory of the Yang--Baxter equation.

Before continuing, we explain the notions of \emph{nondegeneracy} and \emph{involutivity}. A solution $(X,r)$ is called \emph{involutive} if $r^2 = \id_{X^2}$. Furthermore, writing $r(x,y) = (\lambda_x(y),\rho_y(x))$, a solution is called \emph{nondegenerate} if the maps $\lambda_x$ and $\rho_x$ are bijections of $X$ for all $x \in X$.

\begin{convention}
    From now on, by a \emph{solution}, we will mean a set-theoretic solution to the Yang--Baxter equation that is finite, nondegenerate and involutive.
\end{convention}

Rump proved that each solution on a set $X$ corresponds to a \emph{cycle set} structure on $X$, and vice versa \cite{Rump_Decomposition}. We give the definition of a cycle set and outline the Rump correspondence.

\begin{defn}
    A \emph{cycle set} is a pair $(X,\ast)$ where $X$ is a set and $\ast: X \times X \to X$ ; $(x,y) \mapsto x \ast y$ is a binary operation that satisfies the following axioms:
    \begin{align}
       \textnormal{ the map } \sigma_x: X \to X; & \ y \mapsto x \ast y \textnormal{ is bijective for } x \in X\tag{C1} \label{eq:c1} \\
        (x \ast y) \ast (x \ast z) & = (y \ast x) \ast (y \ast z) \ (x,y,z \in X). \tag{C2} \label{eq:c2}
    \end{align}
    Furthermore, we call a cycle set \emph{nondegenerate} if
    \begin{equation} \label{eq:c3}
        \textnormal{the \emph{diagonal map} } T_X: X \to X; \ x \mapsto x \ast x \textnormal{ is a bijection.} \tag{C3}
    \end{equation}
\end{defn}

By a result of Rump \cite{Rump_Decomposition}, for each solution $(X,r)$ with $r(x,y) = (\lambda_x(y), \rho_y(x))$, the operation $x \ast y = \lambda_x^{-1}(y)$ turns $X$ into a nondegenerate cycle set. Vice versa, given a nondegenerate cycle set operation $\ast$ on $X$, the map $r(x,y) = (\sigma_x^{-1}(y),\sigma_x^{-1}(y) \ast x)$ defines a solution on $X$. This correspondence is bijective. 

\begin{conventions}
\begin{enumerate}
    \item Throughout this article, when writing \emph{cycle set}, we will mean \emph{finite, nondegenerate cycle set}.
    \item When referring to results in the literature that use the language of solutions, we will in general translate them into the language of cycle sets. This is justified by Rump's correspondence between solutions and cycle sets.
\end{enumerate}
\end{conventions}

Note that by a result of Rump \cite[Theorem 2]{Rump_Decomposition}, each finite cycle set satisfies \eqref{eq:c3}, so the attribute \emph{nondegenerate} is actually redundant.

\medskip

In the following, a cycle set $X$ is called \emph{squarefree} if its diagonal is trivial, that is, $T_X = \id_X$. Furthermore, $X$ is said to be retractable if there are distinct $x,x' \in X$ such that $x \ast y = x' \ast y$ for all $y \in X$.

In \cite{GI_strong_conjecture}, Gateva-Ivanova poses the following conjecture on squarefree solutions that, in terms of cycle sets, is stated as follows.
\begin{conjecture} \label{conj:retraction_conjecture}
    If $X$ is a squarefree cycle set with $|X| > 1$, then $X$ is retractable.
\end{conjecture}

Gateva-Ivanova's conjecture -- also called the \emph{strong conjecture} -- was ultimately refuted by Vendramin \cite{Vendramin_extension} who developed a method to construct certain coverings of cycle sets that has later been conceptualized by Lebed and Vendramin in form of a cohomology theory for braided sets \cite{Lebed_Vendramin_Homology}.

However, the \emph{retraction problem} -- that is, the problem of finding criteria for retractability of cycle sets -- is still relevant, as a retractable cycle set can, in principle, be constructed from a smaller cycle set - its \emph{retraction} - as an extension \cite{Vendramin_extension}. The definition of a retraction is given in \cref{sec:preliminaries}.

In \cite{Rump_Quasilinear}, Rump addresses the retraction problem for \emph{translation-invariant} and \emph{quasilinear cycle sets}. Here, a cycle set $X$ is called \emph{translation-invariant} if $X$ is squarefree and admits an action by a regular, abelian automorphism group $A$. Identifying $X$ with $A$, translation-invariance translates to $0 \ast 0 = 0$, together with the invariance condition
\begin{equation}
    (a + c) \ast (b + c) = (a \ast b) + c \quad (a,b,c, \in A).
\end{equation}
On the other hand, a cycle set structure on an abelian group $A$ is termed \emph{quasilinear} if it satisfies the condition
\begin{equation}
    a \ast (b + c) = a \ast b + (a - b) \ast c.
\end{equation}
For another, equivalent, definition in the language of $\tau$-groups, see \cref{sec:quasilinear_cycle_sets}. Rump showed that both translation-invariant and quasilinear cycle sets can be obtained as certain \emph{$\tau$-groups}, which are abelian groups $A$ with a bijection $\tau: A \to A$ such that $\tau(0) = 0$. Furthermore, using the language of $\tau$-groups, he establishes a bijective correspondence between translation-invariant cycle sets and \emph{rational} quasilinear cycle sets. Details on the connections between $\tau$-groups and cycle sets can be found in \cite{Rump_Quasilinear}.

For quasilinear cycle sets, Rump poses the following conjecture.

\begin{conjecture} \label{conjecture:retraction_conjecture_intro}
    If $A$ is a quasilinear cycle set with $|A| > 1$, then $A$ is retractable.
\end{conjecture}

If true, \cref{conjecture:retraction_conjecture_intro} also implies the following statement for cycle sets of size $p$, posed by Rump as another conjecture.

\begin{conjecture} \label{conj:size_p_is_retractable}
    If $A$ is a quasilinear cycle set of size $p$, where $p$ is a prime, then there is an isomorphism of quasilinear cycle sets $A \cong \Z_p$, where $x \ast y = \alpha \cdot y$ for some fixed $\alpha \in \Z_p \setminus \{ 0 \}$.
\end{conjecture}

The goal of this article is to prove \cref{conjecture:retraction_conjecture_intro}. The proof uses the theory of left braces that has been developed by Rump for studying cycle sets \cite{Rump_braces}, together with the cabling method of Lebed, Ram\'irez and Vendramin \cite{LRV_Cabling}. The main new ingredient in the proof is the \emph{extended permutation group} of a quasilinear cycle set, a modification of the classical permutation group \cite{ESS_YangBaxter} that also encodes the underlying abelian group.

We give an overview of the structure of this work.

In \cref{sec:preliminaries}, we recall some facts concerning braces, cycle sets and the cabling method. In order to adjust the cabling method to our needs, we reformulate it in terms of the permutation brace of a solution, as opposed to the method in \cite{LRV_Cabling} which uses the structure brace.

In \cref{sec:quasilinear_cycle_sets}, we give a more convenient, but equivalent, definition of quasilinear cycle sets in terms of $\tau$-groups. Moreover, we recall the necessary definitions and results from \cite{Rump_Quasilinear} on quasilinear cycle sets and give Rump's reformulation of \cref{conjecture:retraction_conjecture_intro} in terms of the \emph{socle} of a quasilinear cycle set.

In \cref{sec:extended_permutation_group}, we continue by defining the \emph{extended permutation group} $\Tilde{\G}(A)$ of a quasilinear cycle set $A$. One of its most important properties is its solvability (\cref{thm:tilde_G_is_solvable}), which will be essential in the proof of \cref{conj:size_p_is_retractable}.

In \cref{sec:extended_lambda_action}, we prove the existence of an action - the \emph{$\Tilde{\lambda}$-action} - of $\Tilde{\G}(A)$ by automorphisms of $(\G(A),+)$, the additive group of the permutation brace of a quasilinear cycle set, thereby extending the $\lambda$-action of the permutation brace $\G(A)$ (\cref{thm:extended_lambda_action}). While $\Tilde{\G}(A)$ does not admit a compatible brace structure, it does act as a transitive subgroup of the holomorph of $(\G(A),+)$ (\cref{pro:tilde_g_is_transitive_affine}).

\medskip

Finally, in \cref{sec:solution_of_rumps_conjecture}, we prove Rump's conjecture with the machinery developed in the previous sections. While it is possible to give a proof of Rump's conjecture that uses the extended permutation group $\Tilde{\G}(A)$ while avoiding the use of the $\Tilde{\lambda}$-action, we believe that it allows a more streamlined approach and is of independent interest.

We remark that we do not address the retraction problem for translation-invariant cycle sets which has also been investigated by Rump.

\medskip

\paragraph{\textbf{Funding}} The author is supported by the FWO Senior Postdoctoral Fellowship 1244625N. This work was partially supported by the project OZR3762 of Vrije Universiteit Brussel.

\medskip

\paragraph{\textbf{AI disclosure statement}} During the preparation of the manuscript, ChatGPT has exclusively been used for proofreading the manuscript in the final stage. The mathematical ideas and arguments presented here have been developed without AI assistance. 

\paragraph{\textbf{Acknowledgements}} I would like to thank Ilaria Colazzo and Leandro Vendramin for encouraging me to work on the retraction conjecture for quasilinear cycle sets and for helpful discussions during the early stages of this project.

\section{Preliminaries} \label{sec:preliminaries}

In this section, we collect terminology and results that will be needed in this work. We will provide references for more detailed discussions and restrict ourselves to give exact references only for nontrivial results.

\subsection{Brace theory}

In this article, we make extensive use of the theory of \emph{left braces} (which, in the more general theory of skew braces, are also labelled \emph{skew left braces of abelian type}). For details, we refer the reader to \cite{CJO_Braces,Guarnieri_Vendramin}.

\begin{defn}
    A \emph{left brace} (or \emph{brace}, for short) is a triple $(A,+,\circ)$ where $A$ is a set and $+,\circ: A \times A \to A$ are binary operations such that $(A,\circ)$ is a group, called its \emph{multiplicative group}, $(A,+)$ is an abelian group, called the \emph{additive group}, and such that for all $a,b,c \in A$:
    \begin{equation} \label{eq:brace_law}
        a \circ (b + c) + a = a \circ b + a \circ c.
    \end{equation}
\end{defn}

\begin{conventions}\begin{enumerate}
        \item The additive identity of a brace $A$ is denoted by $0$. Note that this coincides with the multiplicative identity of $A$, which can be verified by plugging $b = c = 0$ into \eqref{eq:brace_law}. For an element $a \in A$, its additive inverse is denoted by $-a$ and its multiplicative inverse by $a^{-1}$.

        \item The trivial subgroup of a - not necessarily abelian - group is always denoted by $0$.
    \end{enumerate}
\end{conventions}

Note that every abelian group $B = (B,+)$ gives rise to a \emph{trivial} brace, meaning that $a \circ b = a + b$ for all $a,b \in B$. We denote this brace $(B,+,+)$ by $\mathrm{Triv}(B)$.

In a brace, the multiplicative group acts by automorphisms of the additive group by means of the \emph{$\lambda$-action}: for each $a \in A$, one defines $\lambda_a: A \to A$ ; $b \mapsto a \circ b - a$. It can be proven \cite[Lemma 1]{CJO_Braces} that this provides a well-defined group homomorphism
\begin{equation}
    \lambda: A \to \Aut(A,+) ; \quad a \mapsto \lambda_a,
\end{equation}
meaning that $\lambda_a(b+c) = \lambda_a(b) + \lambda_a(c)$ and $\lambda_{a \circ b}(c) = \lambda_a(\lambda_b(c))$ for all $a,b,c \in A$.

\medskip

We now recall substructures of a brace.

If $B$ is a brace, a \emph{subbrace} is a subset $A \subseteq B$ such that $A$ is a subgroup of both $(B,+)$ and $(B,\circ)$. This implies that $A$ itself is a brace under the operations inherited from $B$. Furthermore, if $B$, $C$ are braces, we call a map $f: B \to C$ a brace \emph{homomorphism} if $f$ is a homomorphism between the additive and the multiplicative groups.

A \emph{left ideal} of a brace $B$ is a subgroup $I \leq (B,+)$ that is invariant under the $\lambda$-action, that is for all $a \in B$, $b \in I$, we have $\lambda_a(b) \in I$. As the multiplication in a brace can be expressed as $a \circ b = a + \lambda_a(b)$, each left ideal is a subbrace.

\begin{convention}
    From now on, if not stated otherwise, all braces considered are assumed to be finite.
\end{convention}

For a brace $B$, we denote the additive order of an element $a \in B$ by $o_+(a)$. Then for any prime $p$, the \emph{$p$-primary component} is defined as
\[
B_p = \{ a \in A : o_+(a) = p^k \textnormal{ for some } k \geq 0 \}.
\]
As $B_p$ is a characteristic subgroup of $(B,+)$, it is invariant under the $\lambda$-action. Therefore, $B_p$ is a left ideal in $B$. Similarly, we see that the \emph{$p'$-primary component},
\[
B_{p'} = \{ a \in B : p \nmid o_+(a) \},
\]
is a left ideal of $B$. We will use the notion of primary components also when talking about subgroups of abelian groups that do not necessarily come with a brace structure.

\medskip

We will make use of the following statement that is also proven in \cite{CO_SquarefreeIndecomposable}.

\begin{pro} \label{pro:normal_subgroups_acting_trivial}
    Let $A$ be a brace, and let $I,J \leq A$ be left ideals with $I \cap J = 0$. Suppose that $M$ is a normal subgroup of $(A,\circ)$ with $M \subseteq I$. Then for all $a \in M$, $b \in J$, we have $\lambda_a(b) = b$.
\end{pro}

\begin{proof}
    Let $a \in M$, $b \in J$. Then as $M$ is normal, there is an $a' \in M$ with $a \circ b = b \circ a'$. We calculate:
    \[
    b + \lambda_b(a') = b \circ a' = a \circ b = a + \lambda_a(b).
    \]
    As $\lambda_b(a') \in I$, $\lambda_a(b) \in J$, the condition $I \cap J = 0$ implies $a = \lambda_b(a')$ and $b = \lambda_a(b)$, thus proving the proposition.
\end{proof}

\medskip

\subsection{Cycle sets}

By Rump's correspondence (see \cref{sec:introduction}), the theory of (nondegenerate) \emph{cycle sets} is equivalent to the theory of (involutive, nondegenerate) solutions to the Yang--Baxter equation and has the advantage of being more compact, in that their definition requires only one binary operation. Details can be found in \cite{Rump_Decomposition}.

\medskip

If $f: X \to Y$ is a map between cycle sets, $f$ is called a \emph{homomorphism} if $f(x \ast y) = f(x) \ast f(y)$. Derived notions, such as \emph{isomorphism}, are defined in the usual way.

A cycle set $X = (X, \ast)$ is called \emph{trivial} if for all $x,y \in X$, we have $x \ast y = y$.

\medskip

An important property of cycle sets investigated in this article is their retractability. We elaborate on that notion. Given a cycle set $X$, one defines on $X$ the equivalence relation $x \sim y \Leftrightarrow \sigma_x = \sigma_y$. The relation $\sim$ is, in fact, a congruence relation on $X$ (\cite[Lemma 2]{Rump_Decomposition}), therefore $\Ret(X) = X/\!\!\sim$ is again a cycle set under the operation $[x] \ast [y] = [x \ast y]$. The cycle set $\Ret(X)$ is called the \emph{retraction} of $X$. Moreover, the canonical projection $X \twoheadrightarrow \Ret(X)$; $x \mapsto [x]$ is a homomorphism of cycle sets.

If the canonical projection $X \twoheadrightarrow \Ret(X)$ is an isomorphism (or equivalently, if the equivalence relation $\sim$ is trivial), we say that $X$ is \emph{irretractable}. For finite cycle sets, this can also be stated as an equality $|X| = |\Ret(X)|$. A cycle set is \emph{retractable} if it is not irretractable.

\medskip

We now recall the connection between brace theory and cycle sets.

If $B$ is a brace, the underlying set $B$ has a canonical cycle set structure that is given by the operation $g \ast h = \lambda_g^{-1}(h)$ (\cite[Lemma 2]{CJO_Braces}).
Therefore, each brace gives rise to a cycle set.

Vice versa, there are several ways to associate a brace to a cycle set. For this article, the following is relevant:

Let $X = (X,\ast)$ be a cycle set. One defines the \emph{permutation group} of $X$ as the subgroup
\[
\G(X) = \left\langle \sigma_x : x \in X \right\rangle \leq \Sym_X.
\]
 For an element $x \in X$, define the element $\lambda_x \in \G(X)$ by $\lambda_x = \sigma_x^{-1}$.

 The following result is of particular importance in the study of cycle sets (resp. set-theoretic solutions) by means of brace theory:

 \begin{thm} \label{thm:GX_is_a_brace}
     For a cycle set $X$, there is a unique way to define an addition $+$ on $\G(X)$ such that $\G(X,+,\circ)$ is a brace, where $\circ$ is the composition of permutations, and such that for all $x,y \in X$,
     \begin{equation} \label{eq:brace_law_on_GX}
         \lambda_x \circ \lambda_y = \lambda_x + \lambda_{\lambda_x(y)}.
     \end{equation}
 \end{thm}

 In other words, \cref{thm:GX_is_a_brace} says that $\G(X,\circ)$ can be uniquely equipped with a brace structure such that the map $X \to \G(X)$; $x \mapsto \lambda_x$ is a homomorphism of cycle sets, that is,
 
 \begin{equation} \label{eq:lambda_is_hom_of_cycle_sets}
    \lambda_{x \ast y} = \lambda_{\lambda_x}^{-1}(\lambda_y).
 \end{equation}

 \begin{convention}
     Throughout this article, we will use $\lambda_g$ to denote both the action of an element $g \in \G(X)$ on the set $X$ via the corresponding permutation in $\Sym_X$, and its action on $\G(X)$ by means of the $\lambda$-action. There is no harm in doing so, as it will always be clear which set is acted upon.
 \end{convention}

Using the convention above, \eqref{eq:brace_law_on_GX} is equivalent to $\lambda_{\lambda_x}(\lambda_y) = \lambda_{\lambda_x(y)}$. A repeated application of this law shows that for all $g \in \G(X)$, $x \in X$,
\begin{equation} \label{eq:lambda_g_of_lambda_x}
    \lambda_g(\lambda_y) = \lambda_{\lambda_g(y)}.
\end{equation}

\medskip

\subsection{Cabling}

The cabling machinery is a powerful tool that has been developed by Lebed, Ram\'irez and Vendramin \cite{LRV_Cabling} in order to investigate how the diagonal map of a cycle set influences its decomposition behaviour. We remark that rudiments of cabling can already be found in Dehornoy's work about RC-calculus on cycle sets \cite{Dehornoy_RC}. Recently, Colazzo and Van Antwerpen have developed a noninvolutive analogue \cite{colazzo_van_antwerpen}.

In contrast to Lebed, Ram\'irez and Vendramin, we use the permutation brace instead of the structure brace of a cycle set to set up the cabling machinery. However, it is not difficult to see that the resulting calculi are equivalent.

\begin{defn}
    Let $X = (X,\ast)$ be a cycle set and let $k$ be an integer. The \emph{$k$-cabling} of $X$ is then defined as $X^{[k]} = (X, \subcab{k})$ where
    \[
    x \subcab{k} y = \lambda_{k\lambda_x}^{-1}(y),
    \]
    where the product $k\lambda_x$ is formed with respect to the addition in the permutation brace $\G(X)$.
\end{defn}

\begin{thm} \label{thm:cabling_is_a_cycle_set}
    Let $X$ be a cycle set. Then $X^{[k]}$ is a cycle set for each integer $k$.
\end{thm}

\begin{proof}
    \eqref{eq:c1} is quickly checked, and as $X$ is finite, we only need to check \cref{eq:c2}. For $x,y,z \in X$, we have
        \begin{align*}
        (x \subcab{k} y) \subcab{k} (x \subcab{k} z) & = \lambda_{k\lambda_x}^{-1}(y) \subcab{k} \lambda_{k\lambda_x}^{-1}(z) \\
        & = \lambda^{-1}_{k\lambda_{k\lambda_x}^{-1}(\lambda_y)} (\lambda_{k\lambda_x}^{-1}(z)) \\
        & = \lambda^{-1}_{k\lambda_x \circ k\lambda_{k\lambda_x}^{-1}(\lambda_y)} (z) \\
        & = \lambda^{-1}_{k\lambda_x \circ \lambda^{-1}_{k\lambda_x}(k\lambda_y)}(z) \\
        & = \lambda^{-1}_{k\lambda_x + k\lambda_y}(z).
    \end{align*}
    As the last expression is symmetric in $x$ and $y$, \cref{eq:c2} is confirmed.
\end{proof}

\begin{thm} \label{thm:cabling_iteration}
    If $X$ is a cycle set with diagonal $T = T_X$, then for each $k \geq 0$ the operations $\subcab{k}$ satisfy the following iteration:
    \begin{align*}
        x \subcab{0} y & = y \\
        x \subcab{k+1} y & = T^k(x) \ast (x \subcab{k} y) \quad (k \geq 0).
    \end{align*}
\end{thm}

\begin{proof}
    This follows from the representation $k \lambda_x = \lambda_x \circ \lambda_{T(x)} \circ \lambda_{T^2(x)} \circ \ldots \circ \lambda_{T^{k-1}(x)}$ (\cite[Eq.(2.3)]{LRV_Cabling}).
\end{proof}

\begin{thm} \label{thm:permutation_group_of_k_cabling}
    For each cycle set $X$ and any integer $k$, $\G(X^{[k]}) = k\G(X)$. Moreover, the permutation brace structure on $\G(X^{[k]})$ coincides with the brace structure on $k\G(X)$.
\end{thm}

\begin{proof}
    As the definition given in \cite{LRV_Cabling} uses the structure brace instead of the permutation brace, we give a short proof.

    Note that for $x,y \in X$,
    \[
    k\lambda_x \circ k\lambda_y = k\lambda_x + \lambda_{k\lambda_x}(k\lambda_y) \overset{\eqref{eq:lambda_g_of_lambda_x}}{=} k\lambda_x + k\lambda_{\lambda_{k\lambda_x}(y)},
    \]
    which implies the equality of subgroups,
    \[
    \G(X^{[k]}) = \langle k\lambda_x : x \in X \rangle_{\circ}  = \langle k\lambda_x : x \in X \rangle_+ = k\langle \lambda_x : x \in X \rangle = k\G(X).
    \]
    We now check that $X^{[k]} \to k\G(X)$, $x \mapsto k\lambda_x$ is a cycle set homomorphism. For $x,y \in X$, we use \eqref{eq:lambda_g_of_lambda_x} to calculate
    \[
    \lambda_{k\lambda_x}^{-1}(k\lambda_y) = k\lambda_{\lambda_{k\lambda_x}^{-1}(y)} = k \lambda_{x \subcab{k} y}.
    \]
    Thus, by \cref{thm:GX_is_a_brace}, the brace structures on $\G(X^{[k]})$ and $k\G(X)$ coincide.
\end{proof}

\begin{defn}
    The \emph{Dehornoy class} $d = d_X$ of a cycle set $X$ is the smallest integer $d \geq 1$ such that $X^{[d]}$ is a trivial cycle set. 
\end{defn}

\begin{thm} \label{thm:dehornoy_class_is_additive_exponent}
    The Dehornoy class $d_X$ of a cycle set $X$ is the exponent of the additive group $(\G(X),+)$.
\end{thm}

\begin{proof}
    \cite[Theorem G]{LRV_Cabling}.
\end{proof}

\section{Quasilinear cycle sets} \label{sec:quasilinear_cycle_sets}

In this section, we recall Rump's theory of \emph{quasilinear cycle sets} which were introduced in order to investigate cycle sets with certain invariance properties coming from an abelian group. The material in this section can be found in Rump's article \cite{Rump_Quasilinear}, and as before, we only provide exact references for nontrivial statements.

Recall that a \emph{$\tau$-group} is a pair $(A,\tau)$ where $A$ is an abelian group and $\tau: A \to A$ is a bijection with $\tau(0) = 0$. Using this terminology, one definition of quasilinear cycle sets can be given.

\begin{defn}
    A \emph{quasilinear cycle set} is a $\tau$-group $A = (A,\tau)$ such that $A$ is a cycle set under the binary operation $\subtau: A \times A \to A$
    \begin{equation} \label{eq:def_ast_tau}
    a \subtau b = \tau(b - a) - \tau(-a).
    \end{equation}
\end{defn}

\begin{conventions}
    \begin{enumerate}
        \item If $\tau$ is clear from the context, we write $\ast = \subtau$ for the associated cycle set operation.
        \item All $\tau$-groups, and therefore all quasilinear cycle sets, are considered to be finite.
    \end{enumerate}
\end{conventions}

\begin{rem}
    We note here that \eqref{eq:def_ast_tau} establishes a correspondence between $\tau$-groups with underlying group $A$, and binary operations $A \times A \to A$; $(a,b) \mapsto a \ast b$ satisfying the \emph{quasilinearity} condition $a \cdot (b + c) = a \cdot b + (a-b) \cdot c$ \cite[Proposition 1]{Rump_Quasilinear}.
\end{rem}

An \emph{ideal} of a quasilinear cycle set $A$ is a subgroup $B \leq A$ such that $\tau(a + B) = \tau(a) + B$ for all $a \in A$. Given an ideal, the factor group $\Bar{A} = A/B$ is again a quasilinear cycle set with the $\tau$-map $\Bar{\tau}(a+B) = \tau(a) + B$.

Let $A$ be a quasilinear cycle set. Then its \emph{socle} is the subset
\[
\Soc(A) = \{ b \in A \ : \ \forall a \in A: \tau(a+b) = \tau(a) + \tau(b) \},
\]
and its \emph{fixator} is defined as the subset
\[
\Fix(A) = \{ b \in A \ : \ \forall a \in A: \tau(a+b) = \tau(a) + b \}.
\]

The socle of a quasilinear cycle set is an ideal \cite[Proposition 9]{Rump_Quasilinear}. This is also true for $\Fix(A)$ if $A$ is finite \cite[p.155]{Rump_Quasilinear}. Note that the socle and the fixator of a quasilinear cycle set are not to be confused with the socle and the fixator of a brace \cite{skew_left_nilpotent}.

\begin{thm} \label{thm:socle_and_retraction_of_quasilinear_cycle_set}
    The retraction $\Ret(A)$ of a quasilinear cycle set $A$ is again quasilinear. More precisely, the retraction map $A \twoheadrightarrow \Ret(A)$ is equivalent to the canonical map $A \twoheadrightarrow A/\Soc(A)$.
\end{thm}

\begin{proof}
    \cite[Proposition 12]{Rump_Quasilinear}.
\end{proof}

We will later make use of the following observation by Rump \cite[Proposition 5]{Rump_Quasilinear}.

\begin{pro} \label{pro:fix_is_in_socle}
    For each quasilinear cycle set $A$, we have the inclusion $\Fix(A) \subseteq \Soc(A)$. 
\end{pro}

\begin{proof}
    If $b \in \Fix(A)$, then for all $a \in A$, we have $b = (-a) \ast b = \tau(a+b) - \tau(a)$, implying that $\tau(a+b) = \tau(a)+b$. As $\tau(0) = 0$, this implies $\tau(b) = b$ and thus, $\tau(a+b) = \tau(a) + \tau(b)$. Therefore, $b \in \Soc(A)$.
\end{proof}

Due to \cref{thm:socle_and_retraction_of_quasilinear_cycle_set}, \cref{conjecture:retraction_conjecture_intro} can also be rephrased as follows.

\begin{conjecture} \label{conjecture:retraction_conjecture}
    If $A$ is a quasilinear cycle set with $|A| > 1$, then $\Soc(A) \neq 0$.
\end{conjecture}

In the following sections, we give a full proof of \cref{conjecture:retraction_conjecture}.

\section{The extended permutation group of a quasilinear cycle set} \label{sec:extended_permutation_group}

Our main tool to prove \cref{conjecture:retraction_conjecture} will be the \emph{extended permutation group} of a quasilinear cycle set, which is a matched product of the classical permutation group $\G(A)$ and the group $A$. This will on one hand have the advantage of taking into account the structure of the group $A$ while on the other hand facilitating the investigation of quasilinear cycle sets by means of a permutation group that is transitive, unlike the permutation group $\G(A)$ that is, in general, not transitive for a quasilinear cycle set $A$.

For an abelian group $A$ and an element $a \in A$, we denote the translation by $a$ as $\gamma_a: A \to A$; $b \mapsto a + b$. This defines the regular abelian permutation group
\[
\Gamma_A = \{ \gamma_a : a \in A \} \leq \Sym_A.
\]

Using this notation, the $\sigma$-maps in a quasilinear cycle set can be checked to be expressible as
\begin{equation} \label{eq:express_sigma_by_gamma}
    \sigma_a = \gamma_{\tau(-a)}^{-1} \circ \tau \circ \gamma_a^{-1}
\end{equation}

We now prove a lemma that allows us to rewrite products of elements in $\Gamma_A$ and $\G(A)$.
\begin{lem} \label{lem:interchange_sigma_and_gamma}
    Let $A$ be a quasilinear cycle set. Then for all $a,b \in A$, we have
    \begin{align}
        \sigma_a \circ \gamma_b & = \gamma_{a \ast b} \circ \sigma_{a-b}, \label{eq:switch_sigma_gamma} \\
        \gamma_b \circ \sigma_a & = \sigma_{-\tau^{-1}(\tau(-a)-b)} \circ \gamma_{-a-\tau^{-1}(\tau(-a)-b)}. \label{eq:switch_gamma_sigma}
    \end{align}
\end{lem}

\begin{proof}
    Let $a,b \in A$. We use \eqref{eq:express_sigma_by_gamma} to calculate
    \begin{align*}
        \sigma_a \circ \gamma_b & = \gamma_{\tau(-a)}^{-1} \circ \tau \circ \gamma_{a-b}^{-1} \\
        & = \gamma_{\tau(-a)-\tau(-(a-b))}^{-1} \circ \gamma_{\tau(-(a-b))}^{-1} \circ  \tau \circ \gamma_{a-b}^{-1} \\
        & = \gamma_{\tau(-a)-\tau(-(a-b))}^{-1} \circ \sigma_{a-b} \\
        & = \gamma_{a \ast b} \circ \sigma_{a-b}.
    \end{align*}
    On the other hand,
    \begin{align*}
        \gamma_b \circ \sigma_a & =  \gamma_{\tau(-a)-b}^{-1} \circ \tau \circ \gamma_a^{-1} \\
        & = \gamma_{\tau(-a)-b}^{-1} \circ \tau \circ \gamma_{\tau^{-1}(\tau(-a)-b)}  \circ \gamma_{-a-\tau^{-1}(\tau(-a)-b)} \\
        & = \sigma_{-\tau^{-1}(\tau(-a)-b)} \circ \gamma_{-a-\tau^{-1}(\tau(-a)-b)}.
    \end{align*}
\end{proof}

We can now show that $\G(A)$ and $\Gamma_A$ are factors of a matched product.

\begin{thm} \label{thm:tilde_G_is_a_group}
    If $A$ is a quasilinear cycle set, then $\Tilde{\G}(A) = \Gamma_A \circ \G(A)$ is a transitive subgroup of $\Sym_A$. Furthermore, $\Gamma_A \cap \G(A) = 0$ and $\Gamma_A \circ \G(A) = \G(A) \circ \Gamma_A$.
\end{thm}

\begin{proof}
    Let $\Sigma_A = \{ \sigma_a : a \in A\}$. Using \cref{lem:interchange_sigma_and_gamma}, we obtain the equality of sets 
    \begin{equation} \label{eq:Sigma_A_Gamma_A}
        \Sigma_A \circ \Gamma_A = \Gamma_A \circ \Sigma_A.
    \end{equation}
    As the group $\G(A)$ is finite, it is generated by $\Sigma(A)$ as a monoid, so \eqref{eq:Sigma_A_Gamma_A} shows that 
    \[
    \G(A) \circ \Gamma_A = \langle \Sigma_A \rangle \circ \langle \Gamma_A \rangle = \langle \Gamma_A \cup \Sigma_A \rangle = \langle \Gamma_A \rangle \circ \langle \Sigma_A \rangle = \Gamma_A \circ \G(A) = \Tilde{\G}(A).
    \]
    This proves that $\Tilde{\G}(A)$ is indeed a subgroup of $\Sym_A$. As $\Tilde{\G}(A)$ contains the transitive subgroup $\Gamma_A$, the group $\Tilde{\G}(A)$ itself is transitive.
    
    As $\sigma_a(0) = \tau(-a) - \tau(-a) = 0$ for all $a \in A$, we infer that $\G(A)$ fixes the element $0 \in A$. As $\gamma_0$ is the only element of $\Gamma_A$ fixing $0$, it follows that $\Gamma_A \cap \G(A) = 0$.
\end{proof}

\begin{defn}
    The \emph{extended permutation group} of a quasilinear cycle set $A$ is the permutation group
    \[
    \Tilde{\G}(A) = \Gamma_A \circ \G(A) \leq \Sym_A.
    \]
\end{defn}

We now show that, like the permutation group of a cycle set, the extended permutation group of a quasilinear cycle set is also solvable.

In order to do so, we first investigate the behaviour of quasilinear cycle sets under cabling:

\begin{pro} \label{pro:cabling_of_quasilinear_cycle_sets}
    Let $A = (A,\tau)$ be a quasilinear cycle set and $k \geq 0$. Then for all $a,b \in A$, we have
    \begin{equation} \label{eq:cable_tau}
        a \subcab{k} b = a \subtaucab{k} b
    \end{equation}
    In particular, $(A,\tau^k)$ is a quasilinear cycle set.
\end{pro}

\begin{proof}
    Write $T = T_A$ for the diagonal map of the cycle set $A$. Then $T(a) = (\tau(a-a)) - \tau(-a) = -\tau(-a)$. A standard induction then shows that $T^k(a) = -\tau^k(-a)$ for $k \geq 0$.

    We now prove \cref{eq:cable_tau} by induction. Note that $\tau^0 = \id_A$, so $a \subtaucab{0} b = (b-a)-(-a) = b = a \subcab{0} b$. Now assume that $a \subcab{k} b = a \subtaucab{k} b$ holds for all $a,b \in A$ where $k \geq 0$ is fixed, then using \cref{thm:cabling_iteration}, we calculate
    \begin{align*}
        a \subcab{k+1} b & = T^k(a) \ast (a \subcab{k} b) \\
        & = (-\tau^k(-a)) \ast (\tau^k(b-a) - \tau^k(-a)) \\ 
        & = \tau((\tau^k(b-a) - \tau^k(-a)) - (-\tau^k(-a))) - \tau(\tau^k(-a)) \\
        & = \tau^{k+1}(b-a) - \tau^{k+1}(-a) = a \subtaucab{k+1} b. 
    \end{align*}
    By induction, this proves \eqref{eq:cable_tau} for all $k \geq 0$.
\end{proof}

A consequence of \cref{pro:cabling_of_quasilinear_cycle_sets} is that the Dehornoy class of a quasilinear cycle is determined by its $\tau$-map in a simple way.

\begin{cor} \label{cor:dehornoy_class_is_order_of_tau}
    The Dehornoy class $d_A$ of a quasilinear cycle set $A = (A,\tau)$ coincides with $o(\tau)$, the order of $\tau$ as a permutation.
\end{cor}

\begin{proof}
    A quasilinear cycle set $A$ is trivial if and only if $\tau = \id_A$. Using \cref{pro:cabling_of_quasilinear_cycle_sets}, we see that $k = o(\tau)$ is the smallest integer such that $a \subcab{k} b = a \subtaucab{k} b = b$ for all $a,b \in A$. Therefore, $o(\tau) = d_A$.
\end{proof}

Suppose that $n \geq 1$ is an integer and $p$ is a prime. Then it is well known that there is a unique factorization $n = p^v \cdot m$ with $v \geq 0$ and $p \nmid m$. We define $n_p = p^v$ and $n_{p'} = m$.

\begin{pro} \label{pro:sylows_from_cabling}
    Let $A$ be a quasilinear cycle set with Dehornoy class $d = d_A$. Then the $p$-primary component of the brace $\G(A)$ is given by
    \begin{equation} \label{eq:sylows_from_cabling}
        \G(A)_p = \G(A^{[d_{p'}]}).
    \end{equation}
    On the other hand, the $p'$-primary component of $\G(A)$ is given by
    \begin{equation} \label{eq:sylows_from_cabling_complement}
        \G(A)_{p'} = \G(A^{[d_p]}).
    \end{equation}
\end{pro}

\begin{proof}
    It is well-known that if $B$ is a finite abelian group of exponent $n$, then its $p$-primary component is given by $B_p = n_{p'}B$ and its $p'$-primary component is given by $B_{p'} = n_pB$. As $d$ is the exponent of $(\G(A),+)$ (\cref{thm:dehornoy_class_is_additive_exponent}), it follows from \cref{thm:permutation_group_of_k_cabling} and \cref{pro:cabling_of_quasilinear_cycle_sets} that
    \[
    \G(A)_p = d_{p'}\G(A) = \G(A^{[d_{p'}]})
    \]
    and
    \[
    \G(A)_{p'} = d_p\G(A) = \G(A^{[d_p]}).
    \]
\end{proof}

\begin{thm} \label{thm:tilde_G_is_solvable}
    $\Tilde{\G}(A)$ is solvable for each quasilinear cycle set $A$.
\end{thm}

Before going into the proof of this theorem, we recall that for a group $G$, two subgroups $A,B \leq G$ are said to be \emph{permutable} if $A \circ B = B \circ A$.

\begin{proof}
    Let $d = d_A$ be the Dehornoy class of $A$ and let $\mathcal{P}$ be the set of prime divisors of $|\G(A)|$. By \cref{pro:sylows_from_cabling}, we have $\G(A)_p = \G(A^{[d_{p'}]})$ for all $p \in \mathcal{P}$. As $\G(A)_p$ is a left ideal in $\G(A)$, we see that for any $p,q \in \mathcal{P}$,
    \[
    \G(A)_p \circ \G(A)_q = \G(A)_p + \G(A)_q = \G(A)_q + \G(A)_p = \G(A)_q \circ \G(A)_p.
    \]
    Therefore, any two subgroups $\G(A)_p,\G(A)_q \leq \G(A)$ are permutable. Now \cref{thm:tilde_G_is_a_group} additionally shows that for all $p \in \mathcal{P}$,
    \[
    \G(A)_p \circ \Gamma_A = \G(A^{[d_{p'}]}) \circ \Gamma_A = \Gamma_A \circ \G(A^{[d_{p'}]}) = \Gamma_A \circ \G(A)_p.
    \]
    Therefore, the subgroups $\G(A)_p \leq \Tilde{\G}(A)$ ($p \in \mathcal{P}$) and $\Gamma_A \leq \Tilde{\G}(A)$ are mutually permutable.

    Using the fact that the $\G(A)_p$ are left ideals of $\G(A)$, we see that
    \[
    \Tilde{\G}(A) = \Gamma_A \circ \G(A) = \Gamma_A \circ \bigoplus_{p \in \mathcal{P}} \G(A)_p = \Gamma_A \circ \prod_{p \in \mathcal{P}} \G(A)_p.
    \]
    As $\Tilde{\G}(A)$ is generated by the mutually permutable nilpotent subgroups  $\G(A)_p$ ($p \in \mathcal{P}$) and $\Gamma_A$, it follows from the Kegel-Wielandt theorem \cite[Satz 1]{Kegel_Produkte} that $\Tilde{\G}(A)$ is solvable.
\end{proof}

We note that \cref{thm:tilde_G_is_solvable} is already sufficient to prove a special case of \cref{conjecture:retraction_conjecture}.

\begin{cor}
    Let $p$ be a prime and let $(\Z_p,\tau)$ be a quasilinear cycle set structure on $\Z_p$. Then there is an $\alpha \in \Z_p\setminus \{ 0 \}$ such that $\tau(a) = \alpha \cdot a$ for all $a \in \Z_p$.
\end{cor}

\begin{proof}
    By \cref{thm:tilde_G_is_a_group} and \cref{thm:tilde_G_is_solvable}, the extended permutation group $\Tilde{\G}(\Z_p,\tau)$ is solvable and transitive. A classical theorem of Galois \cite[Satz 3.6]{huppert1983endliche} now shows that there is a normal, and therefore unique, $p$-Sylow subgroup $P \leq \Tilde{\G}(\Z_p,\tau)$ of size $p$ such that $\Tilde{\G}(\Z_p,\tau) \leq N_{\Sym_{\Z_p}}(P)$, its normalizer in $\Sym_{\Z_p}$. It follows that $P = \Gamma_{\Z_p}$ and moreover, that $\tau = \sigma_0 \in \Tilde{\G}(\Z_p,\tau)$ normalizes $\Gamma_{\Z_p}$. As $\tau(0) = 0$, this implies that $\tau \in \Aut(\Z_p)$ and thus, $\tau$ is given by $\tau(a) = \alpha \cdot a$ for some $\alpha \in \Z_p\setminus \{ 0 \}$.
\end{proof}

\section{Extended \texorpdfstring{$\lambda$}{lambda}-actions} \label{sec:extended_lambda_action}

Let $A$ be a quasilinear cycle set. We proceed by constructing an action of $\Tilde{\G}(A)$ on $(\G(A),+)$ - the \emph{$\Tilde{\lambda}$-action} - that extends the $\lambda$-action of $(\G(A),\circ)$ on $(\G(A),+)$.

First of all, denote the action of an element $g \in \Tilde{\G}(A)$ on $A$ by $\Tilde{\lambda}_g$. Concretely, this means, that for $g \in \G(A)$ and $a,b \in A$, we have
\[
\Tilde{\lambda}_{\gamma_a \circ g}(b) = a + \lambda_g(b).
\]
The goal of this section is the proof of the following theorem.

\begin{thm} \label{thm:extended_lambda_action}
    For a quasilinear cycle set $A$, there is a unique and well-defined homomorphism $\Tilde{\lambda}: \Tilde{\G}(A) \to \Aut(\G(A),+)$ such that
    \begin{equation} \label{eq:extended_lambda_action}
        \Tilde{\lambda}_g(\lambda_a) = \lambda_{\Tilde{\lambda}_g(a)} \quad (g \in \Tilde{\G}(A), a \in A).
    \end{equation}
\end{thm}

This gives rise to the following definition.

\begin{defn} \label{defn:lambda_tilde_action}
    The \emph{$\Tilde{\lambda}$-action} associated with a quasilinear cycle set $A$ is the unique homomorphism $\Tilde{\lambda}: \Tilde{\G}(A) \to \Aut(\G(A),+)$ such that \eqref{eq:extended_lambda_action} is satisfied.
\end{defn}

In order to prove \cref{thm:extended_lambda_action}, and thus verifying the existence of the $\Tilde{\lambda}$-action, we first rewrite \cref{eq:express_sigma_by_gamma} to express the $\lambda$-maps of the cycle set $A$ as
\begin{equation}
    \lambda_a = \gamma_a \circ \tau^{-1} \circ \gamma_{\tau(-a)}.
\end{equation}
Inverting \cref{eq:switch_sigma_gamma}, we obtain the identity $\gamma_{-b} \circ \lambda_a = \lambda_{a-b} \circ \gamma_{-(a \ast b)}$ for $a,b \in A$. Replacing $b$ by $-a$, $a$ by $b$ and checking that $(a + b) \ast a = -(b \ast (-a))$, this can be rewritten in the form 
\begin{equation} \label{eq:switch_gamma_lambda}
    \gamma_a \circ \lambda_b = \lambda_{a+b} \circ \gamma_{(a+b) \ast a}.
\end{equation}

We use \eqref{eq:switch_gamma_lambda} to prove the following lemma.

\begin{lem} \label{lem:switch_gamma_with_lambda_sum}
    Let $a \in A$ and suppose that $b_1,b_2\ldots,b_k \in A$ for some $k \geq 1$. Then there is an $a' \in A$ such that
    \begin{equation}\label{eq:switch_gamma_with_lambda_sum}
        \gamma_a \circ (\lambda_{b_1} + \lambda_{b_2} + \ldots + \lambda_{b_k}) = (\lambda_{a+b_1} + \lambda_{a+b_2} + \ldots + \lambda_{a+b_k}) \circ \gamma_{a'}.
    \end{equation}
\end{lem}

\begin{proof}
    For $k = 1$, this follows from \eqref{eq:switch_gamma_lambda}. Before continuing, we observe that for $a,b,c \in A$, we have
    \[
    (a+b) \ast (a+c) = b \ast c + (a + b) \ast a.
    \]
    which implies that
    \begin{equation}\label{eq:ba_ca}
        \lambda_{a+b}(b \ast c + (a + b) \ast a) = a+c.
    \end{equation}
    
    Suppose now that $k \geq 1$ and that \eqref{eq:switch_gamma_with_lambda_sum} has been proven for $k$ summands. For $k+1$ summands, we then calculate
    \begin{align*}
        \gamma_a \circ (\lambda_{b_1} + \lambda_{b_2} + \ldots + \lambda_{b_{k+1}}) & = \gamma_a \circ \lambda_{b_1} \circ \lambda_{\lambda_{b_1}}^{-1}(\lambda_{b_2} + \ldots + \lambda_{b_{k+1}}) \\
        & = \lambda_{a+b_1} \circ \gamma_{(a+b_1) \ast a} \circ (\lambda_{b_1 \ast b_2} + \ldots + \lambda_{b_1 \ast b_{k+1}}) \\
        & = \lambda_{a+b_1} \circ (\lambda_{b_1 \ast b_2 + (a+b_1) \ast a} + \ldots + \lambda_{b_1 \ast b_{k+1} + (a+b_1) \ast a} ) \circ \gamma_{a'} \\
        & = (\lambda_{a+b_1} + \lambda_{\lambda_{\lambda_{a+b_1}}(b_1 \ast b_2 + (a+b_1) \ast a)} + \ldots + \lambda_{\lambda_{\lambda_{a+b_1}}(b_1 \ast b_{k+1} + (a+b_1) \ast a)} ) \\
        & \quad \circ \gamma_{a'} \\
        & \overset{\textnormal{\eqref{eq:ba_ca}}}{=} (\lambda_{a+b_1} + \lambda_{a+b_2} + \ldots + \lambda_{a+b_{k+1}}) \circ \gamma_{a'}
    \end{align*}
    with some $a' \in A$. By induction, the statement is now proven.
\end{proof}

\begin{proof}[Proof of \cref{thm:extended_lambda_action}]
    We begin by showing the existence of a unique and well-defined map $\mu: A \to \Aut(\G(A),+)$ with $\mu_a(\lambda_b) = \lambda_{a+b}$ for all $a,b \in A$. Uniqueness follows from the fact that $(\G(A),+)$ is generated by the elements $\lambda_a$ ($a \in A$). In order to prove existence, we first define $\mu_a(g) \in \G(A)$ for $a \in A$, $g \in \G(A)$ by the rule
    \begin{equation} \label{eq:defining_mu}
        \gamma_a \circ g = \mu_a(g) \circ \gamma_{a'}
    \end{equation}
    with a suitable $a' \in A$. This is well-defined as $\Gamma_A \cap \G(A) = 0$ (\cref{thm:tilde_G_is_a_group}). As a consequence of \eqref{eq:switch_gamma_lambda} (or \cref{lem:switch_gamma_with_lambda_sum}), it follows that for $a,b \in A$,
    \begin{equation} \label{eq:mu_a_lambda_b}
        \mu_a(\lambda_b) = \lambda_{a+b}
    \end{equation}

    Now let $a \in A$ and $g,h \in \G(A)$. As $\G(A)$ is finite, we can write $g = \lambda_{b_1} + \ldots + \lambda_{b_k}$ and $h = \lambda_{c_1} + \ldots + \lambda_{c_l}$ for some $b_1,\ldots,b_k,c_1,\ldots,c_l \in A$. An application of \cref{lem:switch_gamma_with_lambda_sum} shows that for a suitable $a' \in A$, we have
    \[
    \gamma_a \circ (\lambda_{b_1} + \ldots + \lambda_{b_k}) = (\lambda_{a+b_1} + \ldots + \lambda_{a+b_k}) \circ \gamma_{a'},
    \]
    which shows that $\mu_a(g) = \lambda_{a+b_1} + \ldots + \lambda_{a+b_k}$. Similarly, $\mu_a(h) = \lambda_{a+c_1} + \ldots + \lambda_{a+c_l}$ and $\mu_a(g+h) =  \lambda_{a+b_1} + \ldots + \lambda_{a+b_k} + \lambda_{a+c_1} + \ldots + \lambda_{a+c_l}$, which proves that $\mu_a(g+h) = \mu_a(g) + \mu_a(h)$.

    We can now finish the proof of the theorem. Let
    \[
    \Tilde{\lambda}: \Tilde{\G}(A) \to \Aut(\G(A),+) ;\  \gamma_a \circ g \mapsto \Tilde{\lambda}_{\gamma_a \circ g} = \mu_a \circ \lambda_g \ (a \in A, g \in \G(A)).
    \]

    As the composition of group endomorphisms, each $\Tilde{\lambda}_{h}$ ($h \in \Tilde{\G}(A)$) is indeed an endomorphism of $(\G(A),+)$. That these are indeed automorphisms follows from $\mu_a(\lambda_b) = \lambda_{a+b}$ for $a,b\in A$ (\cref{eq:switch_gamma_lambda}) and the fact that $\G(A)$ is generated by the elements $\lambda_b$ ($b \in A$).
    
    We now verify \eqref{eq:extended_lambda_action}: for $h = \gamma_a \circ g \in \Tilde{\G}(A)$ ($a \in A$, $g \in \G(A)$), we calculate for $b \in A$ that
    \[
    \Tilde{\lambda}_h(\lambda_b) = \mu_a ( \lambda_g (\lambda_b) )
    \overset{\eqref{eq:lambda_is_hom_of_cycle_sets}}{=} \mu_a (\lambda_{\lambda_g(b)}) \overset{\eqref{eq:mu_a_lambda_b}}{=}
    \lambda_{a+\lambda_g(b)} = \lambda_{\Tilde{\lambda}_h(b)}.
    \]
    For $h,h' \in \Tilde{\G}(A)$, $a \in A$, this now implies
    \[
    \Tilde{\lambda}_{h \circ h'}(\lambda_a) = \lambda_{\Tilde{\lambda}_{h \circ h'}(a)} = \lambda_{\Tilde{\lambda}_h (\Tilde{\lambda}_{h'}(a)) } = \Tilde{\lambda}_h (\Tilde{\lambda}_{h'}(\lambda_a)).
    \]
    It follows from the fact that  $\G(A)$ is generated by the elements $\lambda_b$ ($b \in A$) that $\Tilde{\lambda}_{h \circ h'} =  \Tilde{\lambda}_h \circ \Tilde{\lambda}_{h'}$, thus proving the theorem.
    \end{proof}

As a corollary of the proof, we obtain the following relation.

\begin{cor} \label{cor:tilde_gamma_a}
    If $A$ is a quasilinear cycle set, then for all $g \in \G(A)$, $a \in A$, there is an $a' \in A$ such that
    \begin{equation} \label{eq:tilde_gamma_a}
        \gamma_a \circ g = \Tilde{\lambda}_{\gamma_a}(g) \circ \gamma_{a'}.
    \end{equation}
\end{cor}
    
\begin{rem}
    We note that the $\Tilde{\lambda}$-action does not come from a brace structure on $\Tilde{\G}(A)$, in general. More precisely, there is no brace structure on $\Tilde{\G}(A)$ such that $\G(A)$ is a subbrace of $\Tilde{\G}(A)$ and such that $\lambda_g(h) = \Tilde{\lambda}_g(h)$ for $g \in \Tilde{\G}(A)$, $h \in \G(A)$. Observe first that in such a case, $\G(A)$ would be a left ideal.

    Consider the quasilinear cycle set $A = \Z_{15}$ with $\tau = (4,9,14)$. As $d_A = o(\tau) = 3$ (\cref{cor:dehornoy_class_is_order_of_tau}), we see that the only primes dividing $\Tilde{\G}(A)$ are $3$ and $5$. Suppose that $\Tilde{\G}(A)$ can be equipped with a brace structure such that $\G(A)$ becomes a left ideal of $\Tilde{\G}(A)$. Then $\Tilde{\G}(A)$ has to act trivially on the (additive) quotient $Q = \Tilde{\G}(A) / \G(A)$, as for an abelian group $B$ of size $15$, $|\Aut(B)| = 8$ is not divisible by $3$ or $5$.

    But then the quotient map $\Tilde{\G}(A) \twoheadrightarrow \mathrm{Triv}(Q)$; $g \mapsto g + \G(A)$, where $Q$ is the trivial brace on the abelian group $Q$, can be checked to be a brace homomorphism whose kernel is $\G(A)$. However, $\G(A)$ is not normal in $\Tilde{\G}(A)$, as $\gamma_1 \circ \sigma_0 \circ \gamma_1^{-1} = (0,5,10) \not\in \G(A)$.
\end{rem}

We finally show that, although the $\Tilde{\lambda}$-action is not the $\lambda$-action of a brace structure on $\Tilde{\G}(A)$, it gives rise to something related, namely a transitive affine action of $\Tilde{\G}(A)$ on $(\G(A),+)$. This result is not needed in the following but we believe it to be of independent interest.

\begin{pro} \label{pro:tilde_g_is_transitive_affine}
    For each quasilinear cycle set $A$, the group $\Tilde{\G}(A)$ acts transitively on $\G(A)$ via
    \begin{equation} \label{eq:affine_action}
        \alpha: \Tilde{\G}(A) \times \G(A) \to \G(A);\ (\gamma_a \circ g, c) \mapsto \alpha_{\gamma_a \circ g}(c) = \Tilde{\lambda}_{\gamma_a}(g \circ c) \ (a \in A, g \in \G(A)).
    \end{equation}
\end{pro}

\begin{proof}
    We rewrite \eqref{eq:affine_action} as
    \begin{equation} \label{eq:rewrite_affine_action}
    \alpha_{\gamma_a \circ g}(c) = \Tilde{\lambda}_{\gamma_a}(g + \lambda_g(c)) = \Tilde{\lambda}_{\gamma_a}(g) + \Tilde{\lambda}_{\gamma_a \circ g}(c).
    \end{equation}
    A straightforward calculation shows that $\alpha_0(g) = g$ for all $g \in \G(A)$. Now let $a,b \in A$, $g,h \in \G(A)$. By \cref{cor:tilde_gamma_a}, there is an $a' \in A$ such that $h \circ \gamma_a = \gamma_{a'} \circ \Tilde{\lambda}_{\gamma_{a'}}^{-1}(h) =: \gamma_{a'} \circ h'$. With \eqref{eq:rewrite_affine_action}, we now calculate
    \begin{align*}
        \alpha_{\gamma_b \circ h}(\alpha_{\gamma_a \circ g}(c)) & = \Tilde{\lambda}_{\gamma_b}(h) + \Tilde{\lambda}_{\gamma_b \circ h}(\Tilde{\lambda}_{\gamma_a}(g) + \Tilde{\lambda}_{\gamma_a \circ g}(c)) \\
        & = \Tilde{\lambda}_{\gamma_b}(h) + \Tilde{\lambda}_{\gamma_b \circ h \circ \gamma_a}(g) + \Tilde{\lambda}_{\gamma_b \circ h \circ \gamma_a \circ g}(c) \\
        & = \Tilde{\lambda}_{\gamma_b \circ \gamma_{a'}}(h') + \Tilde{\lambda}_{\gamma_b \circ \gamma_{a'} \circ h'}(g) + \Tilde{\lambda}_{\gamma_b \circ h \circ \gamma_a \circ g}(c) \\
        & = \Tilde{\lambda}_{\gamma_b \circ \gamma_{a'}}(h' + \Tilde{\lambda}_{h'}(g))  +  \Tilde{\lambda}_{\gamma_b \circ h \circ \gamma_a \circ g}(c) \\
        & = \Tilde{\lambda}_{\gamma_b \circ \gamma_{a'}}(h' \circ g)  +  \Tilde{\lambda}_{\gamma_b \circ \gamma_{a'} \circ h' \circ g}(c) \\
        & = \alpha_{\gamma_b \circ \gamma_{a'} \circ h' \circ g}(c) = \alpha_{(\gamma_b \circ h) \circ (\gamma_a \circ g)}(c).
    \end{align*}
    Therefore, $\alpha$ does define an action $\Tilde{\G}(A)$ on $\G(A)$. As $\alpha_g(0) = g$ for $g \in \G(A)$, this action is transitive.
\end{proof}

\section{Proof of Rump's retraction conjecture} \label{sec:solution_of_rumps_conjecture}

In this section, we use the machinery developed in the previous sections to provide a full proof of Rump's conjecture.

Before doing so, we prove a simple lemma concerning invariant equivalence relations. If $G \times X \to X$; $(g,x) \mapsto g \cdot x$ is an action of a group $G$ on a set $X$, then we call an equivalence relation $\sim$ on $X$ \emph{$G$-invariant} if for all $g \in G$, $x,y \in X$, the implication $x \sim y \Rightarrow g \cdot x \sim g \cdot y$ is valid.

\begin{lem} \label{lem:regular_abelian_subgroup_fixing_equivalence_relation}
    Let $G \times X \to X$; $(g,x) \mapsto g \cdot x$ be an action of a group $G$ on a set $X$ and let $\sim$ be a $G$-invariant equivalence relation on $X$. Furthermore, let $A \leq G$ be an abelian subgroup that acts transitively on $X$. Moreover, let $a \in A$ be such that there is an $x_0 \in X$ with $a \cdot x_0 \sim x_0$. Then $a \cdot x \sim x$ for all $x \in X$.
\end{lem}

\begin{proof}
    Let $x \in X$. Then $x = b \cdot x_0$ for some $b \in A$, therefore $a \cdot x = a \cdot (b \cdot x_0) = b \cdot (a \cdot x_0) \sim b \cdot x_0 = x$.
\end{proof}

We can now give a positive answer to \cref{conjecture:retraction_conjecture}.

\begin{thm} \label{thm:rumps_conjecture}
    If $A$ is a quasilinear cycle set with $|A| > 1$, then $\Soc(A) \neq 0$.
\end{thm}

\begin{proof}
    If $\tau = \id_A$, we have $\Soc(A) = A$. Therefore, we will for the rest of the proof assume that $\tau \neq \id_A$, with the consequence that $\G(A) \neq 0$.

    Let $0 \neq M \unlhd \G(A)$ be a minimal normal subgroup. As $\G(A)$ is solvable, $M$ is a $p$-subgroup for some prime $p$. As $M$ is normal, Sylow's theorems imply that $M$ is contained in all Sylow $p$-subgroups of $\G(A)$. In particular, $M \leq \G(A)_p$. As $\G(A)_p \cap \G(A)_{p'} = 0$, \cref{pro:normal_subgroups_acting_trivial} implies that $\lambda_g(h) = h$ for all $g \in M$ and $h \in \G(A)_{p'}$.

    Denoting the Dehornoy class of $A$ by $d = d_A$, this implies that for all $g \in M$, $a \in A$,
    \[
    d_p  \lambda_{\lambda_g(a)} = \lambda_g(d_p  \lambda_a) = d_p  \lambda_a.
    \]
    As $M \neq 0$, this shows that the equivalence relation $\sim$ on $A$ that is defined by $a \sim b \Leftrightarrow d_p  \lambda_a = d_p  \lambda_b$ is nontrivial. We claim that it is invariant under the $\Tilde{\lambda}$-action: for $g \in \Tilde{\G}(A)$, we can use \cref{thm:extended_lambda_action} to infer that
    \[
    a \sim b \Rightarrow d_p  \lambda_a = d_p  \lambda_b \Rightarrow \Tilde{\lambda}_g(d_p  \lambda_a) = \Tilde{\lambda}_g(d_p  \lambda_b) \Rightarrow d_p  \lambda_{\Tilde{\lambda}_g(a)} = d_p  \lambda_{\Tilde{\lambda}_g(b)} \Rightarrow \Tilde{\lambda}_g(a) \sim \Tilde{\lambda}_g(b).
    \]
    Now consider the subgroup $\Tilde{\G}(A)_p = \Tilde{\G}(A^{[d_{p'}]}) = \Gamma_A \circ \G(A)_p \leq \Tilde{\G}(A)$. Note that $\Tilde{\G}(A)_p$ does not need to be a $p$-group! We now define
    \[
    S = \{ g \in \Tilde{\G}(A)_p \ :  \ \forall a \in A: a \sim \Tilde{\lambda}_g(a) \} \unlhd \Tilde{\G}(A)_p.
    \]
    If $0 \neq a \in A$ is such that $a + b \sim b$, for some $b \in A$, then $\gamma_a \in S$ by \cref{lem:regular_abelian_subgroup_fixing_equivalence_relation}. As $\sim$ is nontrivial, this shows that $S \neq 0$.

    Choose a minimal normal subgroup $M' \unlhd \Tilde{\G}(A)_p$ with $M' \leq S$. Then $M'$ is a $q$-group for some prime $q$. If $q \neq p$, then $M'$ has to be contained in $\Gamma_{A_q}$, where $A_q$ denotes the $q$-primary component of $A$. This follows after observing that $\Gamma_{A_q}$ is a $q$-Sylow subgroup of $\Tilde{\G}(A)_p$.
    
    Let $\gamma_a \in M'$, where $0 \neq a \in A_q$. Furthermore, let $b \in A$. Then by \cref{cor:tilde_gamma_a}, there is an $a' \in A$ such that
    \[
    \gamma_a \circ d_{p'}\lambda_b \overset{\eqref{eq:tilde_gamma_a}}{=} \Tilde{\lambda}_{\gamma_a}(d_{p'}\lambda_b) \circ \gamma_{a'} = d_{p'}\lambda_{a+b} \circ \gamma_{a'}.
    \]
    It now follows from $\G(A)_p \cap \Gamma_A = 0$ and the normality of $M'$ that $a' \in A_q$ and $d_{p'} \lambda_{a+b} = d_{p'} \lambda_b$. As $\gamma_a \in S$, we also have $d_{p} \lambda_{a+b} = \Tilde{\lambda}_{\gamma_a}(d_p \lambda_b) = d_{p} \lambda_b$ for all $b \in A$. Writing $1 = m d_p + m' d_{p'}$ (Bezout's lemma), we see that for all $b \in A$,
    \[
    \lambda_{a+b} =  m d_p \lambda_{a+b} + m' d_{p'}\lambda_{a+b} = m d_p \lambda_{b} + m' d_{p'}\lambda_{b} = \lambda_b,
    \]
    which implies $0 \neq a \in \Soc(A)$ by \cref{thm:socle_and_retraction_of_quasilinear_cycle_set}.

    It remains to consider the case $q = p$. In that case, we consider the orbits $\mathcal{O}_a = \{ \Tilde{\lambda}_g(a) : g \in M' \}$ ($a \in A$). As $M'$ is normal in $\Tilde{\G}(A)_p$, the set $\{ \mathcal{O}_a : a \in A \}$ is a block system for the action of $\Tilde{\G}(A)_p$ on $A$, where $|\mathcal{O}_a| > 1$ as $M' \neq 0$. As $\lambda_g(0) = 0$ for all $g \in \G(A)_p$, it follows that $\lambda_g(\mathcal{O}_0) = \mathcal{O}_0$ for all $g \in \G(A)_p$. As $M'$ is a $p$-group, it follows that $p$ divides $|\mathcal{O}_0|$. Let
    \[
     F = \{ a \in \mathcal{O}_0  \ : \ \forall g \in \G(A)_p: \lambda_g(a) = a \}
    \]
    be the set of fixed points for the action of $\G(A)_p$ on $\mathcal{O}_0$. As $\G(A)_p$ is a $p$-group, $p$ divides $|\mathcal{O}_0| - |F|$. As $0 \in F$, this implies that $F \setminus \{ 0 \}$ is nonempty. Let $f \in F \setminus \{ 0 \}$, then $f \in \Fix(A^{[d_{p'}]}) \subseteq \Soc(A^{[d_{p'}]})$ (\cref{pro:fix_is_in_socle}). As the retraction classes of $A^{[d_{p'}]}$ are the cosets of $\Soc(A^{[d_{p'}]})$ (\cref{thm:socle_and_retraction_of_quasilinear_cycle_set}), we see that $d_{p'}\lambda_{f+a} = d_{p'}\lambda_a$ for all $a \in A$.

    As $f$ is contained in the orbit of $0$ under the action of $S$, it follows that $0 \sim f$. As $\gamma_f(0) = f$, we see that $\gamma_f \in S$ by \cref{lem:regular_abelian_subgroup_fixing_equivalence_relation}. As this means that $d_p \lambda_{f+a} = \Tilde{\lambda}_{\gamma_f}(d_p \lambda_a) = d_p \lambda_a$, we argue as before with Bezout's lemma together with \cref{thm:socle_and_retraction_of_quasilinear_cycle_set}, that $0 \neq f \in \Soc(A)$.
\end{proof}

The reader might have noticed that in the proof, the $\Tilde{\lambda}$-action has only been used when showing that the retraction classes of $A^{[d_p]}$ form a block system for the action of $\Tilde{\G}(A)$ on $A$, and when rewriting the products $\gamma_a \circ d_{p'}\lambda_b$ with respect to the factorization $\Tilde{\G}(A) = \G(A) \circ \Gamma_A$. These specific steps in the proof could also have been handled by suitable ad hoc arguments using relations like \eqref{eq:switch_gamma_lambda} on the $\lambda$-maps of a suitable cabling of $A$. However, we decided in favour of developing and employing the $\Tilde{\lambda}$-action, since this more conceptual approach provides deeper insights into the structure of quasilinear cycle sets and might also be generalized to other cases of cycle set structures on sets with a group action.

\bibliography{references}
\bibliographystyle{abbrv}

\end{document}